\documentclass[preprint,12pt,authoryear]{elsarticle}

\usepackage[T1]{fontenc}
\usepackage[utf8]{inputenc}
\usepackage{amsmath,amssymb,amsthm,mathtools}
\usepackage{microtype}
\usepackage{tikz}
\usepackage{hyperref}
\usepackage[nameinlink,capitalize]{cleveref}

\numberwithin{equation}{section}

\newtheorem{theorem}{Theorem}[section]
\newtheorem{proposition}[theorem]{Proposition}
\newtheorem{lemma}[theorem]{Lemma}
\newtheorem{corollary}[theorem]{Corollary}

\theoremstyle{definition}
\newtheorem{definition}[theorem]{Definition}

\theoremstyle{remark}
\newtheorem{remark}[theorem]{Remark}

\newcommand{\R}{\mathbb{R}}
\newcommand{\N}{\mathbb{N}}

\newcommand{\calG}{\mathcal{G}}
\newcommand{\calT}{\mathcal{T}}
\newcommand{\tr}{\operatorname{tr}}
\newcommand{\Sym}{\operatorname{Sym}}

\newcommand{\tw}{\operatorname{tw}}
\newcommand{\ELR}{\mathrm{ELR}}

\journal{Journal of Combinatorial Theory, Series B}

\makeatletter
\def\ps@pprintTitle{%
    \let\@oddhead\@empty
    \let\@evenhead\@empty
    \let\@oddfoot\@empty
    \let\@evenfoot\@empty}
\makeatother

\hypersetup{
    colorlinks=true,
    linkcolor=blue,
    citecolor=blue,
    urlcolor=blue
}

\begin{document}

\begin{frontmatter}

\title{Bounded Treewidth and Complete Monotonicity for Scott--Sokal Spanning-Tree Polynomials}

\author[addr1]{Domingos S. P. Salazar\corref{cor1}}
\address[addr1]{Unidade de Educa\c{c}\~ao a Dist\^ancia e Tecnologia,
Universidade Federal Rural de Pernambuco,
52171-900 Recife, Pernambuco, Brazil}
\cortext[cor1]{Corresponding author.}

\begin{abstract}
Scott and Sokal asked for a structural description of the finite graphs \(G\) for which inverse powers \(T_G^{-\beta}\) of the spanning-tree polynomial are completely monotone.  We prove the following bounded-treewidth criterion.  If \(G\) is a finite connected simple graph with \(\operatorname{tw}(G)\le k\), then \(T_G^{-\beta}\) is completely monotone for every
\[
        \beta>\frac{k-1}{2}.
\]
Consequently, every partial \(3\)-tree is covered throughout the first Scott--Sokal open interval \(1<\beta<3/2\), including finite Apollonian networks, \(K_5-e\), and the four-spoke wheel \(W_4\).

The proof combines the real Riesz/Wishart integral for determinants, star--mesh elimination of simplicial vertices, and a Gaussian Laplace kernel for the degree-\(d\) star.  These ingredients produce a Euclidean Laplace representation invariant stable under perfect elimination orderings; general bounded-treewidth graphs follow by chordal completion and monotone deletion of completion edges.  The ordinary Riesz threshold also marks the boundary of the method: in the interval \(1<\beta<3/2\), genuine treewidth-\(4\) cores would require a new endpoint or boundary-Riesz ingredient.
\end{abstract}


\end{frontmatter}

\begingroup
\small
\noindent\textit{Pudim AI disclosure.}
This manuscript was developed with assistance from the Pudim AI research workflow.
Public provenance and APP ledger:
\url{https://github.com/pudim-project/pudim-ai-demo-zetalaw}; APP-0074 and APP-0075.
\par
\endgroup

\medskip

\section{Introduction and motivation}

Let $G=(V,E)$ be a finite connected graph and let
\begin{equation}
    T_G(x)=\sum_{T\in\calT(G)}\prod_{e\in T}x_e
\end{equation}
be its spanning-tree polynomial.  For background on spanning trees and graph Laplacians, see \cite{Bollobas1998,Bapat2010}; for the multivariate Tutte-polynomial viewpoint, see \cite{SokalTutte}.  A central question of Scott and Sokal asks, for a given graph $G$, for which positive exponents $\beta$ the inverse power $T_G^{-\beta}$ is completely monotone on $(0,\infty)^E$ \cite{ScottSokal2014}.  Equivalently, writing $\calG_\beta$ for the class of graphs with this property, one asks for a structural description of $\calG_\beta$.

The broad picture developed in \cite{ScottSokal2014} is sharp in several ranges.  For $\beta\in\{1/2,1,3/2,\ldots\}$ all graphs lie in $\calG_\beta$.  For $0<\beta<1/2$ the class consists of graphs obtained from forests by parallel extension of edges, and for $1/2<\beta<1$ it consists of series-parallel graphs.  The first open interval is therefore
\begin{equation}
        1<\beta<3/2.
\end{equation}
Scott and Sokal explicitly identify two small unresolved test graphs in this interval: the wheel graph with four spokes, $W_4$, and $K_5-e$.
The two graph labellings used below are shown in \cref{fig:test-graphs}.
This paper proves a structural theorem in this interval and beyond.  The main result, \cref{thm:bounded-treewidth}, replaces the parallel-star viewpoint by a standard graph parameter: bounded treewidth controls the complete-monotonicity threshold
\begin{equation}
        \tw(G)\le k
        \quad\Longrightarrow\quad
        T_G^{-\beta}\text{ is completely monotone for every }
        \beta>\frac{k-1}{2}.
\end{equation}
In particular, every partial \(3\)-tree belongs to \(\calG_\beta\) for every \(\beta>1\).  This includes finite Apollonian networks, a standard scale-free planar family from statistical physics, and it also proves the two named Scott--Sokal test graphs \(K_5-e\) and \(W_4\).

The mechanism is simplicial elimination.  Complete graphs are represented by the real Riesz/Wishart integral.  A simplicial vertex is then removed by the star--mesh transform, producing rational clique conductances \(u_i u_j/(u_1+\cdots+u_d)\).  These rational factors are exactly positive Gaussian Laplace kernels once the old clique vertices are interpreted as Euclidean points.  This observation is stable under iteration, provided one uses the stronger invariant of a Euclidean Laplace representation rather than ordinary complete monotonicity alone.

\begin{figure}[t]
\centering
\begin{tikzpicture}[
    scale=0.92,
    vertex/.style={circle,draw,inner sep=1.4pt,minimum size=16pt},
    edge/.style={line width=0.45pt},
    edge label/.style={font=\scriptsize,inner sep=1.2pt}
]
\begin{scope}[xshift=-3.45cm]
    \node[vertex] (s) at (0,1.85) {$s$};
    \node[vertex] (t) at (0,-1.85) {$t$};
    \node[vertex] (one) at (-1.55,0.75) {$1$};
    \node[vertex] (two) at (1.55,0.75) {$2$};
    \node[vertex] (three) at (0,-0.85) {$3$};

    \draw[edge] (one) -- (two);
    \draw[edge] (one) -- (three);
    \draw[edge] (two) -- (three);

    \draw[edge] (s) -- (one);
    \draw[edge] (s) -- (two);
    \draw[edge] (s) -- (three);

    \draw[edge] (t) -- (one);
    \draw[edge] (t) -- (two);
    \draw[edge] (t) -- (three);

    \node[font=\small] at (0,-2.55) {(a) \(K_5-e\)};
\end{scope}

\begin{scope}[xshift=3.45cm]
    \node[vertex] (zero) at (0,0) {$0$};
    \node[vertex] (wone) at (-1.35,1.35) {$1$};
    \node[vertex] (wtwo) at (1.35,1.35) {$2$};
    \node[vertex] (wthree) at (1.35,-1.35) {$3$};
    \node[vertex] (wfour) at (-1.35,-1.35) {$4$};

    \draw[edge] (wone) -- node[edge label,above] {$a$} (wtwo);
    \draw[edge] (wtwo) -- node[edge label,right] {$b$} (wthree);
    \draw[edge] (wthree) -- node[edge label,below] {$c$} (wfour);
    \draw[edge] (wfour) -- node[edge label,left] {$d$} (wone);

    \draw[edge] (zero) -- (wone);
    \draw[edge] (zero) -- (wtwo);
    \draw[edge] (zero) -- (wthree);
    \draw[edge] (zero) -- (wfour);

    \node[edge label] at (-0.78,0.47) {$p$};
    \node[edge label] at (0.78,0.47) {$q$};
    \node[edge label] at (0.78,-0.47) {$r$};
    \node[edge label] at (-0.78,-0.47) {$s$};

    \node[font=\small] at (0,-2.55) {(b) \(W_4\)};
\end{scope}
\end{tikzpicture}
\caption{The two Scott--Sokal test graphs in the notation used below.  Both are partial \(3\)-trees.  In (a), \(K_5-e\) is realized as two nonadjacent apices \(s,t\) joined to the terminal triangle \(1,2,3\); the missing edge is \(st\), with edge variables \(a_i=x_{si}\), \(b_i=x_{ti}\), and \(c_{ij}=x_{ij}\).  In (b), \(W_4\) has hub \(0\), rim cycle \(1-2-3-4-1\), spoke variables \(p,q,r,s\), and rim variables \(a,b,c,d\).}
\label{fig:test-graphs}
\end{figure}
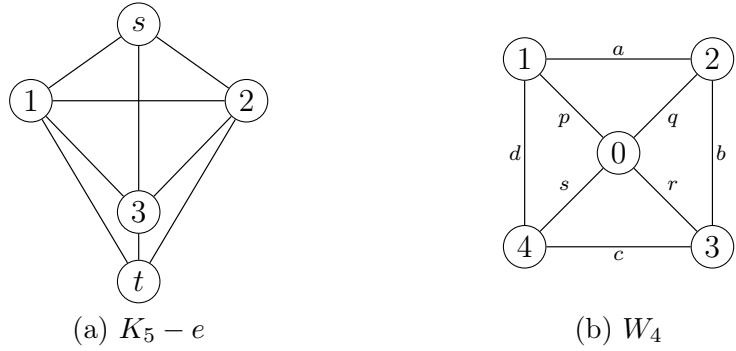

For chordal graphs, a perfect elimination ordering reduces the proof to successive simplicial extensions of complete graphs.  For general graphs of treewidth at most \(k\), one passes to a chordal completion of clique number at most \(k+1\) and then deletes the completion edges by a monotone limit.  The ordinary real Riesz threshold for a clique of size \(k+1\) is precisely \(\beta>(k-1)/2\), which explains the exponent in the theorem.

The theorem does not classify \(\calG_\beta\) in this interval.  By the present method, a genuine treewidth-\(4\) core would require the ordinary threshold \(\beta>3/2\), so the open interval would require a new boundary-Riesz or non-Riesz positivity mechanism.

The paper is organized as follows.  Section~\ref{sec:formalism} fixes complete-monotonicity and graph-polynomial notation.  Section~\ref{sec:riesz-star} proves the Riesz and star--mesh tools, including the degree-\(d\) Gaussian star kernel.  Section~\ref{sec:bounded-treewidth} proves the bounded-treewidth theorem and records the parallel star-lift family as a corollary.  Section~\ref{sec:apollonian} gives the Apollonian-network application, and Sections~\ref{sec:k5e} and \ref{sec:w4} record the named Scott--Sokal test graphs.  Section~\ref{sec:conclusion} records consequences and limitations.

\section{Formalism}
\label{sec:formalism}

\subsection{Complete monotonicity}

We use the standard multivariate version of complete monotonicity and its Laplace-transform interpretation.  Background may be found in Widder's treatment of the Bernstein theorem and in the semigroup formulation of Berg, Christensen and Ressel \cite{Widder,BergChristensenRessel}.

\begin{definition}[Complete monotonicity]
Let $n\ge1$.  For a multi-index \(\alpha=(\alpha_1,\ldots,\alpha_n)\in\N^n\), write \(|\alpha|=\alpha_1+\cdots+\alpha_n\).  A function \(f\in C^\infty((0,\infty)^n)\) is \emph{completely monotone} if, for every such \(\alpha\),
\begin{equation}
        (-1)^{|\alpha|}\partial^\alpha f(x)\ge0,
        \qquad x\in(0,\infty)^n .
\end{equation}
\end{definition}

The following elementary form of the Bernstein--Hausdorff--Widder principle is the only direction used in the proofs \cite[Chapter~IV]{Widder}.

\begin{proposition}[Positive Laplace kernels]
\label{prop:laplace-cm}
Let $\mu$ be a positive Borel measure on $[0,\infty)^n$ such that
\begin{equation}
        f(x)=\int_{[0,\infty)^n} e^{-x\cdot t}\,d\mu(t)
\end{equation}
is finite for every $x\in(0,\infty)^n$.  Then $f$ is completely monotone.
\end{proposition}

\begin{proof}
For each multi-index $\alpha$ and each compact subcone $K\Subset(0,\infty)^n$, the integrand
\begin{equation}
        t^\alpha e^{-x\cdot t}
\end{equation}
is dominated locally by finite linear combinations of Laplace kernels at slightly smaller positive arguments.  After truncating by $t_i\le R$, monotone or dominated convergence justifies passage to the limit.  Thus
\begin{equation}
        (-1)^{|\alpha|}\partial^\alpha f(x)
        =\int_{[0,\infty)^n} t^\alpha e^{-x\cdot t}\,d\mu(t)\ge0.
\end{equation}
\end{proof}

\begin{lemma}[Closure properties]
\label{lem:closure}
Let $f$ and $g$ be completely monotone functions on positive orthants.
\begin{enumerate}
\item If $f$ and $g$ are functions of disjoint variable blocks, then $fg$ is completely monotone in the union of the variables.
\item More generally, if $f(x)=\int e^{-x\cdot s}\,d\mu(s)$ and $g(x)=\int e^{-x\cdot t}\,d\nu(t)$ are positive Laplace transforms in the same variables, then $fg$ is a positive Laplace transform, namely the transform of the convolution of $\mu$ and $\nu$.
\item If $f_y(x)$ is completely monotone for each $y$ and $F(x)=\int f_y(x)\,d\nu(y)$ is finite for a positive measure $\nu$, then $F$ is completely monotone.
\end{enumerate}
\end{lemma}

\begin{proof}
The first two assertions follow either by differentiating products or by multiplying the corresponding Laplace kernels.  The third follows from \cref{prop:laplace-cm} when the $f_y$ are represented by positive Laplace kernels.  Equivalently, one may integrate the inequalities
$(-1)^{|\alpha|}\partial^\alpha f_y\ge0$
against $\nu$ after truncation and apply monotone convergence.
\end{proof}

\subsection{Spanning-tree polynomials and Scott--Sokal graph classes}

Our graph notation is standard.  Unless explicitly noted otherwise, graphs are finite, connected, and simple when treewidth is discussed.  A spanning tree is a connected acyclic spanning subgraph, and the weighted generating function below is the usual spanning-tree polynomial \cite{Bollobas1998,Bapat2010,SokalTutte}.

\begin{definition}[Spanning-tree polynomial]
Let $G=(V,E)$ be a finite connected graph.  Its spanning-tree polynomial is
\begin{equation}
        T_G(x)=\sum_{T\in\calT(G)}\prod_{e\in T}x_e,
        \qquad x=(x_e)_{e\in E}.
\end{equation}
Here $\calT(G)$ denotes the set of spanning trees of $G$.
\end{definition}

Throughout the paper the variables $x_e$ are taken in the positive edge cone $(0,\infty)^E$.  We often call them conductances, following the standard electrical-network convention used for weighted graph Laplacians and star--mesh transformations \cite{DoyleSnell1984,CurtisIngermanMorrow1998}.

\begin{definition}[The Scott--Sokal class]
For $\beta>0$, let $\calG_\beta$ denote the class of finite connected graphs $G$ such that $T_G^{-\beta}$ is completely monotone on $(0,\infty)^E$.
\end{definition}

\begin{proposition}[Matrix-tree determinant form]
\label{prop:matrix-tree}
Fix a root vertex $r\in V$.  For each edge $e=ij$, let $b_e\in\R^{V\setminus\{r\}}$ be the signed incidence vector obtained from $e_i-e_j$ after deleting the root coordinate.  Then
\begin{equation}
        T_G(x)=\det\left(\sum_{e\in E}x_e b_e b_e^T\right).
\end{equation}
\end{proposition}

\begin{proof}
This is Kirchhoff's matrix-tree theorem in weighted form \cite{Kirchhoff,Bapat2010}.  The weighted graph Laplacian is
\begin{equation}
        L_G(x)=\sum_{e\in E}x_e \widetilde b_e\widetilde b_e^T,
\end{equation}
where $\widetilde b_e$ is any oriented incidence vector of $e$ in $\R^V$.  Deleting the row and column indexed by $r$ gives the reduced Laplacian
\begin{equation}
        L_G^{(r)}(x)=\sum_{e\in E}x_e b_e b_e^T .
\end{equation}
Kirchhoff's matrix-tree theorem gives $\det L_G^{(r)}(x)=T_G(x)$.
\end{proof}

\section{Riesz and star--mesh tools}
\label{sec:riesz-star}

\subsection{The real Riesz/Wishart determinant kernel}

Let $\Pi_m$ denote the cone of positive-definite real symmetric $m\times m$ matrices.  We use the standard real symmetric-cone normalization of the Riesz/Wishart integral \cite{FarautKoranyi,Muirhead1982}.  The multivariate gamma function is
\begin{equation}
        \Gamma_m(\alpha)=\int_{\Pi_m} e^{-\tr Y}\det(Y)^{\alpha-(m+1)/2}\,dY,
\end{equation}
finite precisely for $\alpha>(m-1)/2$; explicitly,
\begin{equation}
        \Gamma_m(\alpha)=\pi^{m(m-1)/4}\prod_{j=1}^m \Gamma\left(\alpha-\frac{j-1}{2}\right).
\end{equation}

\begin{lemma}[Riesz/Wishart integral]
\label{lem:riesz}
Let $M\in\Pi_m$ and let $\beta>(m-1)/2$.  Then
\begin{equation}
        \det(M)^{-\beta}
        =\Gamma_m(\beta)^{-1}
          \int_{\Pi_m} e^{-\tr(MY)}\det(Y)^{\beta-(m+1)/2}\,dY .
\end{equation}
\end{lemma}

\begin{proof}
Apply the change of variables $Y=M^{-1/2}ZM^{-1/2}$ in the defining integral for $\Gamma_m(\beta)$.  Then
\begin{equation}
        \tr(MY)=\tr Z,
        \qquad
        \det Y=\det(M)^{-1}\det Z.
\end{equation}
The Lebesgue measure on $\Sym(m,\R)$ transforms as
\begin{equation}
        dY=\det(M)^{-(m+1)/2}\,dZ.
\end{equation}
Therefore
\begin{align}
\int_{\Pi_m} e^{-\tr(MY)}\det(Y)^{\beta-(m+1)/2}\,dY
&=\det(M)^{-\beta}\int_{\Pi_m} e^{-\tr Z}\det(Z)^{\beta-(m+1)/2}\,dZ  \\
&=\Gamma_m(\beta)\det(M)^{-\beta}.
\end{align}
Dividing by $\Gamma_m(\beta)$ proves the identity.
\end{proof}

\begin{corollary}[Determinantal complete monotonicity]
\label{cor:det-cm}
Let $A_1,\ldots,A_n$ be positive-semidefinite real symmetric $m\times m$ matrices and suppose $M(x)=\sum_i x_i A_i$ is positive definite for all $x\in(0,\infty)^n$.  If $\beta>(m-1)/2$, then
\begin{equation}
        x\longmapsto \det M(x)^{-\beta}
\end{equation}
is completely monotone on $(0,\infty)^n$.
\end{corollary}

\begin{proof}
By \cref{lem:riesz},
\begin{equation}
\det M(x)^{-\beta}
=\Gamma_m(\beta)^{-1}\int_{\Pi_m}
\exp\left(-\sum_i x_i\tr(A_iY)\right)
\det(Y)^{\beta-(m+1)/2}\,dY .
\end{equation}
Since $A_i,Y$ are positive semidefinite, $\tr(A_iY)\ge0$.  This is a positive Laplace representation in the variables $x_i$, and \cref{prop:laplace-cm} applies.
\end{proof}

\subsection{Star--mesh elimination}

The next lemma is the spanning-tree-polynomial form of the familiar star--mesh, or $Y$--$\Delta$, elimination from electrical network theory \cite{DoyleSnell1984,CurtisIngermanMorrow1998}.  We state it for a degree-\(d\) eliminated vertex because the bounded-treewidth proof iterates this operation along chordal elimination orderings.

\begin{lemma}[Star--mesh elimination]
\label{lem:starmesh}
Let \(H=(V,E_H)\) be a connected network and let \(S\subseteq V\) have size \(d\ge1\).  Form \(G\) by adjoining a new vertex \(v\) adjacent exactly to the vertices of \(S\), with conductances \(u_i>0\) for \(i\in S\).  Put
\begin{equation}
        U=\sum_{i\in S}u_i,
        \qquad
        p_{ij}(u)=\frac{u_i u_j}{U}
        \quad (i<j,\ i,j\in S).
\end{equation}
Let \(x+p(u)\) denote the edge-weight vector of the network obtained from \(H\) by adding \(p_{ij}(u)\) to the conductance between \(i\) and \(j\) for every distinct \(i,j\in S\), creating that terminal edge if it was absent, and leaving all other edge weights unchanged.  Then
\begin{equation}
        T_G(x,u)=U\,T_H(x+p(u)).
\end{equation}
\end{lemma}

\begin{proof}
Choose a root vertex different from \(v\), and put \(v\) last in the reduced weighted Laplacian of \(G\).  The reduced Laplacian has block form
\begin{equation}
        \begin{pmatrix}
        L_H(x)+D & -w\\
        -w^T & U
        \end{pmatrix},
\end{equation}
where \(D\) and \(w\) encode the conductances incident to \(v\).  Taking the Schur complement of the scalar block \(U\) \cite{HornJohnson2013} gives
\begin{equation}
        \det L_G
        =
        U\det\left(L_H(x)+D-\frac{ww^T}{U}\right).
\end{equation}
The matrix \(D-ww^T/U\) is exactly the reduced Laplacian contribution of the complete graph on \(S\), with edge conductances \(u_i u_j/U\).  These conductances are added to the terminal network, either increasing existing terminal edges or creating absent ones.  The matrix-tree theorem gives the asserted identity.
\end{proof}

For a three-edge star with conductances \(u=(u_1,u_2,u_3)\), we continue to write
\begin{equation}
        p(u)=\left(\frac{u_1u_2}{U},\frac{u_1u_3}{U},\frac{u_2u_3}{U}\right),
        \qquad U=u_1+u_2+u_3,
\end{equation}
with coordinates indexed by \(12,13,23\).

\subsection{The Gaussian star kernel}

\begin{lemma}[Gaussian star kernel]
\label{lem:gaussian-star}
Let \(d\ge1\), let \(\xi_1,\ldots,\xi_d\) be points in a real Euclidean space, and put
\begin{equation}
        r_{ij}=\|\xi_i-\xi_j\|^2,
        \qquad
        U=\sum_{i=1}^d u_i,
        \qquad
        u_i>0.
\end{equation}
If
\begin{equation}
        \beta>\frac{d-1}{2},
\end{equation}
then
\begin{equation}
        K(u)
        =
        U^{-\beta}
        \exp\left(
        -\frac1U\sum_{1\le i<j\le d}u_i u_j r_{ij}
        \right)
\end{equation}
is a positive Laplace transform in \((u_1,\ldots,u_d)\).  More precisely, after embedding the original Euclidean space into a larger Euclidean space if necessary, let \(B\) be any \((d-1)\)-dimensional affine subspace containing \(\xi_1,\ldots,\xi_d\).  Then
\begin{equation}
\begin{aligned}
        K(u)
        &=
        \frac{\pi^{-(d-1)/2}}{\Gamma(\beta-(d-1)/2)}
        \int_0^\infty\int_B
        t^{\beta-(d-1)/2-1}  \\
        &\quad\times
        \exp\left(
        -\sum_{i=1}^d u_i\bigl(t+\|z-\xi_i\|^2\bigr)
        \right)\,dz\,dt .
\end{aligned}
\end{equation}
For \(d=1\), \(B=\{\xi_1\}\) and the empty pair sum is zero.  In particular, for \(d=3\) the kernel used in the parallel star-lift proof is completely monotone for every \(\beta>1\).
\end{lemma}

\begin{proof}
The weighted variance identity gives
\begin{equation}
        \sum_{i=1}^d u_i\|z-\xi_i\|^2
        =
        U\left\|z-\frac{\sum_i u_i\xi_i}{U}\right\|^2
        +
        \frac1U\sum_{1\le i<j\le d}u_i u_j\|\xi_i-\xi_j\|^2 .
\end{equation}
The weighted barycenter belongs to \(B\).  Therefore
\begin{equation}
        \int_B
        \exp\left(-\sum_{i=1}^d u_i\|z-\xi_i\|^2\right)\,dz
        =
        \left(\frac{\pi}{U}\right)^{(d-1)/2}
        \exp\left(
        -\frac1U\sum_{i<j}u_i u_j r_{ij}
        \right).
\end{equation}
Set \(\alpha=\beta-(d-1)/2>0\).  The gamma identity gives
\begin{equation}
        U^{-\alpha}
        =
        \Gamma(\alpha)^{-1}\int_0^\infty t^{\alpha-1}e^{-tU}\,dt .
\end{equation}
Multiplying the two identities gives the displayed representation.  Since all coefficients \(t+\|z-\xi_i\|^2\) are nonnegative, this is a positive Laplace representation in the variables \(u_i\).
\end{proof}

\section{The bounded-treewidth theorem}
\label{sec:bounded-treewidth}

The proof uses a stronger invariant than complete monotonicity.  Its Laplace coordinates are squared Euclidean distances between graph vertices.

\begin{definition}[Euclidean Laplace representation]
Let \(G=(V,E)\) be a finite connected graph and let \(\beta>0\).  We say that \(G\) has a Euclidean Laplace representation at exponent \(\beta\), abbreviated \(\ELR_\beta\), if there exist an integer \(N\ge0\) and a positive Borel measure \(\mu\) on \((\R^N)^V\) such that, for every \(x\in(0,\infty)^E\),
\begin{equation}
        T_G(x)^{-\beta}
        =
        \int_{(\R^N)^V}
        \exp\left(
        -\sum_{ij\in E}x_{ij}\|\xi_i-\xi_j\|^2
        \right)\,d\mu(\xi),
\end{equation}
with the integral finite for every positive \(x\).
\end{definition}

\begin{lemma}[\(\ELR_\beta\) implies complete monotonicity]
\label{lem:elr-cm}
If \(G\) has \(\ELR_\beta\), then \(T_G^{-\beta}\) is completely monotone on \((0,\infty)^E\).
\end{lemma}

\begin{proof}
The displayed formula is a positive Laplace representation in the edge variables, with Laplace coordinates
\begin{equation}
        t_{ij}(\xi)=\|\xi_i-\xi_j\|^2\ge0 .
\end{equation}
The result follows from \cref{prop:laplace-cm}.
\end{proof}

\begin{lemma}[Complete graphs have Euclidean Laplace representations]
\label{lem:complete-graph-elr}
For \(m\ge1\), the complete graph \(K_{m+1}\) has \(\ELR_\beta\) whenever
\begin{equation}
        \beta>\frac{m-1}{2}.
\end{equation}
Equivalently, \(K_r\) has \(\ELR_\beta\) whenever \(\beta>(r-2)/2\).
\end{lemma}

\begin{proof}
Root \(K_{m+1}\) at vertex \(0\), set \(c_0=0\), and set \(c_i=e_i\in\R^m\) for \(1\le i\le m\).  For an edge \(ij\), put \(b_{ij}=c_i-c_j\).  By \cref{prop:matrix-tree},
\begin{equation}
        T_{K_{m+1}}(x)
        =
        \det\left(\sum_{0\le i<j\le m}x_{ij}b_{ij}b_{ij}^T\right).
\end{equation}
For \(Y\in\Pi_m\), choose the positive square root \(Y^{1/2}\) and define
\begin{equation}
        \xi_i(Y)=Y^{1/2}c_i\in\R^m .
\end{equation}
Then
\begin{equation}
        b_{ij}^TYb_{ij}
        =
        \|\xi_i(Y)-\xi_j(Y)\|^2 .
\end{equation}
Inserting the matrix-tree determinant into \cref{lem:riesz} gives
\begin{equation}
\begin{aligned}
        T_{K_{m+1}}(x)^{-\beta}
        &=
        \Gamma_m(\beta)^{-1}
        \int_{\Pi_m}
        \exp\left(
        -\sum_{i<j}x_{ij}\|\xi_i(Y)-\xi_j(Y)\|^2
        \right)  \\
        &\quad\times
        \det(Y)^{\beta-(m+1)/2}\,dY .
\end{aligned}
\end{equation}
This is an \(\ELR_\beta\) representation, after pushing the positive Riesz measure forward by \(Y\mapsto(\xi_i(Y))_{i=0}^m\).
\end{proof}

\begin{proposition}[Simplicial extension preserves \(\ELR_\beta\)]
\label{prop:simplicial-extension-elr}
Let \(H=(V,E_H)\) be a connected graph having \(\ELR_\beta\).  Let \(S\subseteq V\) be a clique of size \(d\ge1\), and let \(G\) be obtained from \(H\) by adjoining a new vertex \(v\) adjacent exactly to the vertices of \(S\).  If
\begin{equation}
        \beta>\frac{d-1}{2},
\end{equation}
then \(G\) has \(\ELR_\beta\).
\end{proposition}

\begin{proof}
Let \(u_i=x_{vi}>0\) for \(i\in S\), and put \(U=\sum_{i\in S}u_i\).  By \cref{lem:starmesh},
\begin{equation}
        T_G(x,u)=U\,T_H(x+p(u)).
\end{equation}
Since \(H\) has \(\ELR_\beta\), there are \(N\) and a positive measure \(\mu\) on \((\R^N)^V\) such that
\begin{equation}
        T_H(y)^{-\beta}
        =
        \int
        \exp\left(
        -\sum_{ij\in E_H}y_{ij}\|\xi_i-\xi_j\|^2
        \right)\,d\mu(\xi).
\end{equation}
Substituting \(y=x+p(u)\) gives
\begin{equation}
\begin{aligned}
        T_G(x,u)^{-\beta}
        &=
        \int
        \exp\left(
        -\sum_{ij\in E_H}x_{ij}\|\xi_i-\xi_j\|^2
        \right)  \\
        &\quad\times
        U^{-\beta}
        \exp\left(
        -\frac1U
        \sum_{\substack{i<j\\ i,j\in S}}
        u_i u_j\|\xi_i-\xi_j\|^2
        \right)\,d\mu(\xi).
\end{aligned}
\end{equation}
For each fixed \(\xi\), the second line is the Gaussian star kernel of \cref{lem:gaussian-star}.

It remains only to record the measure-theoretic construction of the new configuration measure.  Put \(N_0=\max\{N,d-1\}\) and embed \(\R^N\) isometrically into \(\R^{N_0}\).  For each configuration \(\xi\), let \(a(\xi)=d^{-1}\sum_{i\in S}\xi_i\).  Apply Gram--Schmidt, with a fixed deterministic tie-breaking rule using the standard basis of \(\R^{N_0}\), to obtain measurably an orthonormal \((d-1)\)-frame \(e_1(\xi),\ldots,e_{d-1}(\xi)\) whose span contains the vectors \(\xi_i-a(\xi)\), \(i\in S\).  The deterministic tie-breaking is included only to make this a Borel choice of frame, so that the push-forward measure below is unambiguous.  Thus
\begin{equation}
        z_{\xi}(s)
        =
        a(\xi)+\sum_{\ell=1}^{d-1}s_\ell e_\ell(\xi),
        \qquad s\in\R^{d-1},
\end{equation}
parametrizes a \((d-1)\)-dimensional affine subspace containing the clique points.

Define a map to configurations in \((\R^{N_0+1})^{V\cup\{v\}}\) by
\begin{equation}
        \xi'_w=(\xi_w,0)\quad (w\in V),
        \qquad
        \xi'_v=(z_{\xi}(s),\sqrt t).
\end{equation}
Then, for \(i\in S\),
\begin{equation}
        \|\xi'_v-\xi'_i\|^2=t+\|z_{\xi}(s)-\xi_i\|^2,
\end{equation}
and old edge distances are unchanged.  Let \(\nu\) be the push-forward of
\begin{equation}
        d\mu(\xi)\,
        \frac{\pi^{-(d-1)/2}}{\Gamma(\beta-(d-1)/2)}
        t^{\beta-(d-1)/2-1}\,dt\,ds .
\end{equation}
All integrands are nonnegative, so Tonelli's theorem applies.  Substituting the Gaussian star-kernel formula gives
\begin{equation}
        T_G(x,u)^{-\beta}
        =
        \int
        \exp\left(
        -\sum_{ij\in E_H}x_{ij}\|\xi'_i-\xi'_j\|^2
        -\sum_{i\in S}u_i\|\xi'_v-\xi'_i\|^2
        \right)\,d\nu(\xi').
\end{equation}
This is precisely an \(\ELR_\beta\) representation for \(G\).
\end{proof}

\begin{proposition}[Chordal bounded-clique graphs]
\label{prop:chordal-elr}
Let \(H\) be a connected chordal graph with clique number \(\omega(H)\le k+1\).  Then \(H\) has \(\ELR_\beta\) for every
\begin{equation}
        \beta>\frac{k-1}{2}.
\end{equation}
\end{proposition}

\begin{proof}
A chordal graph has a perfect elimination ordering \cite{FulkersonGross1965}.  Equivalently, it may be built by starting with a clique \(K_r\), \(1\le r\le k+1\), and repeatedly adding a new vertex adjacent to a clique \(S\) in the graph already built.  At each extension step, \(|S|\le k\), because the new vertex together with \(S\) is a clique of the final graph.

If \(r=1\), the base graph has spanning-tree polynomial \(1\).  If \(r\ge2\), write \(r=m+1\).  Since \(m\le k\), the assumption \(\beta>(k-1)/2\) implies \(\beta>(m-1)/2\), so \cref{lem:complete-graph-elr} gives \(\ELR_\beta\) for the base clique.  At each subsequent extension, \(|S|=d\le k\), so \(\beta>(k-1)/2\ge(d-1)/2\), and \cref{prop:simplicial-extension-elr} preserves \(\ELR_\beta\).  Induction over the construction proves the proposition.
\end{proof}

\begin{lemma}[Deleting edges preserves \(\ELR_\beta\)]
\label{lem:edge-deletion-elr}
Let \(\widehat G=(V,\widehat E)\) have \(\ELR_\beta\), and let \(G=(V,E)\) be a connected spanning subgraph with \(E\subseteq\widehat E\).  Then \(G\) has \(\ELR_\beta\).
\end{lemma}

\begin{proof}
Write \(F=\widehat E\setminus E\).  Let \(x=(x_e)_{e\in E}\), and give every edge in \(F\) conductance \(\varepsilon>0\).  Since \(\widehat G\) has \(\ELR_\beta\),
\begin{equation}
        T_{\widehat G}(x,\varepsilon\mathbf 1)^{-\beta}
        =
        \int
        \exp\left(
        -\sum_{e\in E}x_e\|\xi_i-\xi_j\|^2
        -\varepsilon\sum_{f\in F}\|\xi_i-\xi_j\|^2
        \right)\,d\mu(\xi).
\end{equation}
As \(\varepsilon\downarrow0\),
\begin{equation}
        T_{\widehat G}(x,\varepsilon\mathbf 1)\longrightarrow T_G(x),
\end{equation}
because the spanning trees using completion edges receive a factor tending to zero, while the spanning trees using only \(E\) are exactly the spanning trees of \(G\).  Since \(G\) is connected, \(T_G(x)>0\).  On the integral side the integrand increases pointwise to
\begin{equation}
        \exp\left(-\sum_{e\in E}x_e\|\xi_i-\xi_j\|^2\right).
\end{equation}
Monotone convergence gives the desired \(\ELR_\beta\) representation for \(G\).
\end{proof}

\begin{theorem}[Main theorem: bounded treewidth]
\label{thm:bounded-treewidth}
Let \(G=(V,E)\) be a finite connected simple graph.  If \(\tw(G)\le k\), then for every \(\beta>0\) satisfying
\begin{equation}
        \beta>\frac{k-1}{2},
\end{equation}
the function \(T_G^{-\beta}\) is completely monotone on \((0,\infty)^E\).
\end{theorem}

\begin{proof}
By the chordal-completion characterization of treewidth \cite{Bodlaender1998}, there is a chordal graph \(\widehat G=(V,\widehat E)\) on the same vertex set such that \(E\subseteq\widehat E\) and \(\omega(\widehat G)\le k+1\).  By \cref{prop:chordal-elr}, \(\widehat G\) has \(\ELR_\beta\) for every \(\beta>(k-1)/2\).  By \cref{lem:edge-deletion-elr}, \(G\) has \(\ELR_\beta\).  Finally, \cref{lem:elr-cm} implies that \(T_G^{-\beta}\) is completely monotone.
\end{proof}

\begin{corollary}[Partial \(3\)-trees]
\label{cor:partial-3-trees}
If \(\tw(G)\le3\), then \(T_G^{-\beta}\) is completely monotone for every \(\beta>1\).  Hence every partial \(3\)-tree belongs to \(\calG_\beta\) throughout the interval \(1<\beta<3/2\).
\end{corollary}

\begin{proof}
Apply \cref{thm:bounded-treewidth} with \(k=3\).
\end{proof}

\begin{corollary}[Parallel star-lifts over triangle and \(K_4\) cores]
\label{cor:parallel-star-lifts}
Let \(G\) be obtained from a core \(R\) of size \(3\) or \(4\) by adjoining any finite set of pairwise nonadjacent vertices, each adjacent to exactly three vertices of \(R\), and assume \(G\) is connected.  Then \(T_G^{-\beta}\) is completely monotone for every \(\beta>1\).
\end{corollary}

\begin{proof}
Add all missing edges inside \(R\).  The resulting graph is chordal with clique number at most \(4\): start with the complete core on \(R\), and add each new vertex as a simplicial vertex adjacent to a clique of size \(3\).  Hence \(\tw(G)\le3\), and \cref{cor:partial-3-trees} applies.  Stronger effective-core connectivity hypotheses are sufficient, but the bounded-treewidth argument only needs the connectedness assumed here.
\end{proof}

\begin{remark}[Boundary of the method]
\label{rem:method-boundary}
The proof is tied to the ordinary real Riesz threshold.  A clique of size \(k+1\) produces a reduced determinant of size \(k\), and the ordinary positive Riesz density requires \(\beta>(k-1)/2\).  Thus, in the interval \(1<\beta<3/2\), the method reaches all partial \(3\)-trees but not genuine treewidth-\(4\) cores.  It does not classify \(\calG_\beta\) in this interval, nor does it imply that every \(K_5\)-minor-free graph belongs to \(\calG_\beta\).  The endpoint \(\beta=(k-1)/2\), including \(\beta=1\) for partial \(3\)-trees, requires a separate boundary-Riesz analysis.
\end{remark}

\section{Application: finite Apollonian networks}
\label{sec:apollonian}

Finite Apollonian networks give a useful infinite application class from statistical-physics graph geometry.  Andrade, Herrmann, Andrade, and da Silva introduced them as scale-free, small-world, Euclidean graphs motivated by physical and infrastructure networks \cite{AndradeHerrmannAndradeDaSilva2005}.  We use that literature only for the graph family: the conclusion below concerns the Kirchhoff determinant, not the distinct Ising, Potts, percolation, or conduction models studied on the same graphs.

A finite Apollonian network is obtained from a plane triangle by repeatedly selecting a triangular face, inserting a new vertex into that face, and joining the new vertex to the three boundary vertices of the face.  These graphs are planar \(3\)-trees, also called stacked triangulations \cite{BergoldEtAl2024}.  Hence every finite Apollonian network \(A\) satisfies
\begin{equation}
        \tw(A)\le3.
\end{equation}

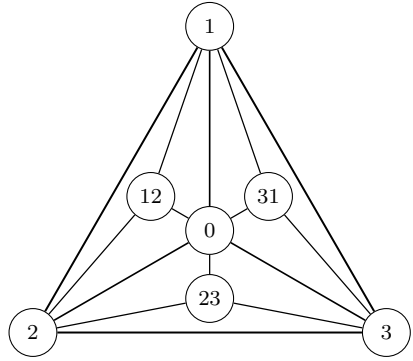
\begin{figure}[h!]
\centering
\begin{tikzpicture}[
    scale=2.25,
    vertex/.style={circle,draw,inner sep=1.2pt,minimum size=18pt,font=\scriptsize},
    genzero/.style={line width=0.75pt},
    genone/.style={line width=0.62pt},
    gentwo/.style={line width=0.48pt}
]
  \node[vertex] (v1) at (0,1.20) {$1$};
  \node[vertex] (v2) at (-1.039,-0.60) {$2$};
  \node[vertex] (v3) at (1.039,-0.60) {$3$};
  \node[vertex] (v0) at (0,0) {$0$};
  \node[vertex] (v12) at (-0.346,0.20) {$12$};
  \node[vertex] (v23) at (0,-0.40) {$23$};
  \node[vertex] (v31) at (0.346,0.20) {$31$};

  \draw[genone] (v0)--(v1);
  \draw[genone] (v0)--(v2);
  \draw[genone] (v0)--(v3);

  \draw[gentwo] (v12)--(v1);
  \draw[gentwo] (v12)--(v2);
  \draw[gentwo] (v12)--(v0);
  \draw[gentwo] (v23)--(v2);
  \draw[gentwo] (v23)--(v3);
  \draw[gentwo] (v23)--(v0);
  \draw[gentwo] (v31)--(v3);
  \draw[gentwo] (v31)--(v1);
  \draw[gentwo] (v31)--(v0);

  \draw[genzero] (v1)--(v2);
  \draw[genzero] (v2)--(v3);
  \draw[genzero] (v3)--(v1);
\end{tikzpicture}
\caption{A second-generation finite Apollonian network.  Each new vertex is inserted into a triangular face and joined to its three boundary vertices.  The nested insertions are planar \(3\)-tree stacking steps beyond a fixed-core parallel star-lift construction.}
\label{fig:apollonian-network}
\end{figure}

\begin{corollary}[Apollonian networks]
\label{cor:apollonian}
Let \(A\) be a finite connected Apollonian network.  Then, for every \(\beta>1\), the function \(T_A^{-\beta}\) is completely monotone on \((0,\infty)^{E(A)}\).  Hence
\begin{equation}
        A\in\calG_\beta
        \qquad\text{for every}\qquad
        1<\beta<\frac32 .
\end{equation}
\end{corollary}

\begin{proof}
Finite Apollonian networks are planar \(3\)-trees, hence have treewidth at most \(3\).  Apply \cref{cor:partial-3-trees}.
\end{proof}

In determinant form, if \(L_A^{(r)}(x)\) is a reduced weighted Laplacian, then \cref{prop:matrix-tree} gives
\begin{equation}
        T_A(x)=\det L_A^{(r)}(x).
\end{equation}
Thus, writing \(q=2\beta\), the rooted fractional Kirchhoff determinant
\begin{equation}
        Z_{q,A}(x)
        :=
        \det L_A^{(r)}(x)^{-q/2}
        =
        T_A(x)^{-q/2}
\end{equation}
is completely monotone for every \(q>2\).  For integer \(q\), this is the rooted \(q\)-component Gaussian-field determinant on \(A\) \cite{LeJan2011}:
\begin{equation}
        \int_{\phi_r=0}
        \exp\left[
        -\frac12\sum_{ij\in E(A)}x_{ij}\|\phi_i-\phi_j\|^2
        \right]\prod_{v\ne r}d\phi_v
        \propto
        \det L_A^{(r)}(x)^{-q/2}.
\end{equation}
For noninteger \(q>2\), complete monotonicity gives a positive substitute for this analytic continuation: for every \(x>0\), there is a probability law \(\mathbb P_x\) on nonnegative edge-energy variables \(S=(S_e)_{e\in E(A)}\) such that, for every \(h\ge0\),
\begin{equation}
        \frac{Z_{q,A}(x+h)}{Z_{q,A}(x)}
        =
        \mathbb E_x\left[
        \exp\left(-\sum_{e\in E(A)}h_eS_e\right)
        \right].
\end{equation}
The point is a positive fractional Kirchhoff determinant theory on Apollonian geometry, not a claim about the different Ising, Potts, percolation, sandpile, resistor-network, or quantum-walk models and observables studied on the same graphs.

\section{Application: \texorpdfstring{$K_5-e$}{K5-e}}
\label{sec:k5e}

The graph \(K_5-e\), labelled as in the left panel of \cref{fig:test-graphs}, is a partial \(3\)-tree, so it is an immediate application of \cref{thm:bounded-treewidth}.  We nevertheless spell out the star--mesh factorization because it is the smallest nontrivial instance of two parallel apices acting on the same core and gives a concrete model for the general mechanism.

Let $G=K_5-e$.  Realize $G$ as two nonadjacent apices $s,t$ joined to a terminal triangle with vertices $1,2,3$; equivalently, the missing edge in the complete graph on $\{s,t,1,2,3\}$ is $st$.  All edge variables in this section are positive.  Write
\begin{equation}
        a_i=x_{si},\qquad b_i=x_{ti},\qquad 1\le i\le3,
\end{equation}
and write
\begin{equation}
        c=(c_{12},c_{13},c_{23})=(x_{12},x_{13},x_{23}).
\end{equation}
Let
\begin{equation}
        A=a_1+a_2+a_3,
        \qquad
        B=b_1+b_2+b_3,
\end{equation}
and for $w=(w_{12},w_{13},w_{23})$ define the triangle spanning-tree polynomial
\begin{equation}
        E(w)=w_{12}w_{13}+w_{12}w_{23}+w_{13}w_{23}.
\end{equation}

\begin{lemma}[Star--mesh factorization for $K_5-e$]
\label{lem:k5e-factor}
With the notation above,
\begin{equation}
        T_{K_5-e}(a,b,c)=AB\,E\bigl(c+p(a)+p(b)\bigr).
\end{equation}
\end{lemma}

\begin{proof}
Apply \cref{lem:starmesh} first to the apex $s$ and then to the apex $t$.  Eliminating $s$ multiplies the spanning-tree polynomial by $A$ and adds the triangle conductances $p(a)$ to the terminal triangle.  Eliminating $t$ multiplies by $B$ and adds $p(b)$.  The remaining graph is the terminal triangle with edge-vector $c+p(a)+p(b)$, whose spanning-tree polynomial is $E$.
\end{proof}

\begin{lemma}[Riesz representation for the triangle]
\label{lem:triangle-riesz}
For $w_{12},w_{13},w_{23}>0$ and $\beta>1/2$,
\begin{equation}
\begin{aligned}
 E(w)^{-\beta}
 &=
 \Gamma_2(\beta)^{-1}
 \int_{\Pi_2}
 \exp\{-w_{12}r_{12}(Y)-w_{13}r_{13}(Y)-w_{23}r_{23}(Y)\} \\
 &\hspace{35mm}\times
 \det(Y)^{\beta-3/2}\,dY,
\end{aligned}
\end{equation}
where
\begin{equation}
        r_{12}(Y)=Y_{11}+Y_{22}-2Y_{12},
        \qquad
        r_{13}(Y)=Y_{11},
        \qquad
        r_{23}(Y)=Y_{22}.
\end{equation}
Moreover, for each $Y\in\Pi_2$, the triple $(r_{12},r_{13},r_{23})$ is a squared-distance triple of three points in $\R^2$.
\end{lemma}

\begin{proof}
Root the triangle at vertex $3$, so the remaining coordinates correspond to vertices $1$ and $2$.  The edge vectors are
\begin{equation}
        b_{12}=e_1-e_2,
        \qquad b_{13}=e_1,
        \qquad b_{23}=e_2.
\end{equation}
Thus
\begin{align}
        M(w)
        &=
        w_{12}(e_1-e_2)(e_1-e_2)^T
        +w_{13}e_1e_1^T+w_{23}e_2e_2^T  \\
        &=
        \begin{pmatrix}
        w_{12}+w_{13} & -w_{12}\\
        -w_{12} & w_{12}+w_{23}
        \end{pmatrix}.
\end{align}
Its determinant is $E(w)$.  Applying \cref{lem:riesz} with $m=2$ gives the displayed integral, because
\begin{equation}
        \tr(M(w)Y)
        =
        w_{12}(Y_{11}+Y_{22}-2Y_{12})
        +w_{13}Y_{11}+w_{23}Y_{22}.
\end{equation}
If $Y\in\Pi_2$, choose $u,v\in\R^2$ whose Gram matrix is $Y$.  Then
\begin{equation}
        r_{12}=|u-v|^2,
        \qquad r_{13}=|u-0|^2,
        \qquad r_{23}=|v-0|^2,
\end{equation}
which proves the squared-distance assertion.
\end{proof}

\begin{corollary}[Application to $K_5-e$]
\label{cor:k5e}
For every $\beta>1$, the function $T_{K_5-e}^{-\beta}$ is completely monotone on its positive edge cone.  Hence $K_5-e\in\calG_\beta$ for every $1<\beta<3/2$.
\end{corollary}

\begin{proof}
The graph \(K_5-e\) is chordal and has maximal cliques of size \(4\), so \(\tw(K_5-e)=3\).  The result follows from \cref{cor:partial-3-trees}.  The explicit factorization above gives the same conclusion directly by the degree-three star-kernel representation.
\end{proof}

\section{Application: the wheel \texorpdfstring{$W_4$}{W4}}
\label{sec:w4}

The wheel \(W_4\), labelled as in the right panel of \cref{fig:test-graphs}, is also a partial \(3\)-tree, hence an immediate application of \cref{thm:bounded-treewidth}.  We retain the explicit \(K_4\)-core star--mesh factorization because it shows how the treewidth theorem specializes to this original Scott--Sokal test graph.

Let $W_4$ have hub $0$ and rim vertices $1,2,3,4$, with rim cycle
\begin{equation}
        1-2-3-4-1
\end{equation}
and spokes from $0$ to each rim vertex.  All edge variables in this section are positive.  Use the following edge variables:
\begin{equation}
        p=x_{01},\quad q=x_{02},\quad r=x_{03},\quad s=x_{04},
\end{equation}
for the spokes, and
\begin{equation}
        a=x_{12},\quad b=x_{23},\quad c=x_{34},\quad d=x_{41}
\end{equation}
for the rim edges.  Put
\begin{equation}
        S=p+a+d.
\end{equation}
We eliminate the rim vertex $1$, whose three incident conductances are $p,a,d$.

\begin{lemma}[Star--mesh factorization for $W_4$]
\label{lem:w4-factor}
Let $K_4$ be the complete graph on vertices $0,2,3,4$.  Assign its edge conductances by
\begin{equation}
        y_{02}=q+\frac{pa}{S},
        \qquad
        y_{04}=s+\frac{pd}{S},
        \qquad
        y_{24}=\frac{ad}{S},
\end{equation}
and
\begin{equation}
        y_{03}=r,
        \qquad
        y_{23}=b,
        \qquad
        y_{34}=c.
\end{equation}
Then
\begin{equation}
        T_{W_4}(p,q,r,s,a,b,c,d)=S\,T_{K_4}(y).
\end{equation}
\end{lemma}

\begin{proof}
Apply \cref{lem:starmesh} to the degree-three rim vertex $1$, with terminals $0,2,4$ and conductances $p,a,d$.  Eliminating vertex $1$ multiplies by $S=p+a+d$ and adds conductances
\begin{equation}
        \frac{pa}{S},\qquad \frac{pd}{S},\qquad \frac{ad}{S}
\end{equation}
between the terminal pairs $02,04,24$, respectively.  The remaining original edges are $02,03,04,23,34$, with conductances $q,r,s,b,c$.  After the added star--mesh edges are included, the remaining graph is $K_4$ on $0,2,3,4$ with the conductances stated above.
\end{proof}

\begin{lemma}[Riesz representation for $K_4$]
\label{lem:k4-riesz}
Root $K_4$ at vertex $3$ and use basis vectors $e_0,e_2,e_4$ for the remaining vertices.  For the six edges of $K_4$ set
\begin{equation}
        b_{03}=e_0,
        \quad b_{23}=e_2,
        \quad b_{43}=e_4,
\end{equation}
\begin{equation}
        b_{02}=e_0-e_2,
        \quad b_{04}=e_0-e_4,
        \quad b_{24}=e_2-e_4.
\end{equation}
Let
\begin{equation}
        M(y)=\sum_{ij} y_{ij} b_{ij}b_{ij}^T .
\end{equation}
Here and below the sum runs over the six edges
\begin{equation}
        03,\ 23,\ 43,\ 02,\ 04,\ 24
\end{equation}
of the complete graph on $\{0,2,3,4\}$, with $43$ denoting the same unoriented edge as $34$.  Then $T_{K_4}(y)=\det M(y)$.  For every $\beta>1$,
\begin{equation}
        T_{K_4}(y)^{-\beta}
        =\Gamma_3(\beta)^{-1}\int_{\Pi_3}
        \exp\left(-\sum_{ij}y_{ij}U_{ij}(Y)\right)
        \det(Y)^{\beta-2}\,dY,
\end{equation}
where
\begin{equation}
        U_{ij}(Y)=b_{ij}^TYb_{ij}.
\end{equation}
Moreover, for each $Y\in\Pi_3$, the triple $(U_{02},U_{04},U_{24})$ is a squared-distance triple of three points in $\R^2$.
\end{lemma}

\begin{proof}
The identity $T_{K_4}(y)=\det M(y)$ is \cref{prop:matrix-tree} applied to $K_4$ rooted at vertex $3$.  The integral representation follows from \cref{lem:riesz} with $m=3$; the condition is $\beta>(3-1)/2=1$.

It remains only to prove the squared-distance assertion.  Since $Y\in\Pi_3$, write $Y=LL^T$.  Then
\begin{equation}
        U_{ij}=b_{ij}^TYb_{ij}=|L^Tb_{ij}|^2.
\end{equation}
In particular,
\begin{equation}
        U_{02}=|L^Te_0-L^Te_2|^2,
        \quad
        U_{04}=|L^Te_0-L^Te_4|^2,
        \quad
        U_{24}=|L^Te_2-L^Te_4|^2.
\end{equation}
These are the squared distances among the three points $L^Te_0,L^Te_2,L^Te_4\in\R^3$.  Any three Euclidean points lie in an affine plane, so the same squared distances are realized by three points in $\R^2$.
\end{proof}

\begin{corollary}[Application to $W_4$]
\label{cor:w4}
For every $\beta>1$, the function $T_{W_4}^{-\beta}$ is completely monotone on its positive edge cone.  Hence $W_4\in\calG_\beta$ for every $1<\beta<3/2$.
\end{corollary}

\begin{proof}
The graph \(W_4\) is a subgraph of \(K_5-e\), obtained by deleting one rim diagonal from a chordal graph with maximal clique size \(4\).  Hence \(\tw(W_4)\le3\).  The result follows from \cref{cor:partial-3-trees}.  The explicit factorization above gives the same conclusion directly by the degree-three star-kernel representation.
\end{proof}

\section{Consequences, scope, and limitations}
\label{sec:conclusion}

\Cref{thm:bounded-treewidth} gives the structural threshold
\begin{equation}
        \tw(G)\le k
        \quad\Longrightarrow\quad
        T_G^{-\beta}\text{ is completely monotone for }
        \beta>\frac{k-1}{2}.
\end{equation}
For the first open Scott--Sokal interval \(1<\beta<3/2\), this proves every partial \(3\)-tree, including finite Apollonian networks and the test graphs \(K_5-e\) and \(W_4\).  These examples are therefore not isolated positive cases; they are members of a full treewidth class.

The theorem does not classify \(\calG_\beta\) in this interval, nor does it imply that all \(K_5\)-minor-free graphs lie in \(\calG_\beta\).  Its natural boundary is the ordinary Riesz threshold: a clique of size \(k+1\) requires \(\beta>(k-1)/2\) in the positive-density argument.  Thus the present method reaches treewidth \(3\) in the first open interval, but a genuine treewidth-\(4\) core would require a new boundary-Riesz representation or a different positivity mechanism.  The endpoint \(\beta=1\) for partial \(3\)-trees is likewise outside the present open-density theorem.

\end{document}